\documentclass[a4paper]{article}
\usepackage{amsmath,amsthm,amssymb,stmaryrd}
\usepackage[all]{xy}
\usepackage{url}

\newtheorem{theorem}{Theorem}[section]
\newtheorem{proposition}[theorem]{Proposition}
\newtheorem{lemma}[theorem]{Lemma}
\newtheorem{corollary}[theorem]{Corollary}
\newtheorem{remark}[theorem]{Remark}
\theoremstyle{definition}

\newtheorem{definition}[theorem]{Definition}

\newcommand{\nat}{\mathbb{N}}

\newcommand{\colim}{\operatorname{colim}}
\newcommand{\sub}{\mathsf{01Sub}}
\newcommand{\czf}{\mathbf{CZF}}
\newcommand{\monadf}[1]{\mathsf{#1}}
\newcommand{\btwo}{\mathbf{2}}
\newcommand{\bthree}{\mathbf{3}}

\newcommand{\names}{\mathbb{A}}
\newcommand{\rint}{[0,1]}
\newcommand{\rmap}{\monadf{R}\text{-}\mathbf{Map}}
\newcommand{\lmap}{\monadf{L}\text{-}\mathbf{Map}}
\newcommand{\supp}{\operatorname{Supp}}
\newcommand{\perma}{\operatorname{Perm}(\names)}

\newcommand{\catc}{\mathbb{C}}

\title{An Algebraic Weak Factorisation System on 01-Substitution Sets:
  A Constructive Proof} \author{Andrew Swan\thanks{This work was
    supported by EPRSC grant EP/K023128/1}}

\begin{document}
\maketitle
\begin{abstract}
  We will construct an algebraic weak factorisation system on the
  category of 01-substitution sets such that the $\monadf{R}$-algebras
  are precisely the Kan fibrations together with a choice of Kan
  filling operation. The proof is based on Garner's small object
  argument for algebraic weak factorization systems. In order to
  ensure the proof is valid constructively, rather than applying the
  general small object argument, we give a direct proof based on the
  same ideas. We use this us to give an explanation why the
  $J$-computation rule is absent from the original cubical set model
  and suggest a way to fix this.
\end{abstract}

\section{Introduction}

\subsection{Aims}

We will construct an algebraic weak factorisation system on $\sub$
such that the $\monadf{R}$-algebras are precisely the Kan fibrations
together with a choice of Kan filling operation. It will be
algebraically free in the sense of \cite{garnersmallobject}. However
rather than applying the result in \cite{garnersmallobject}, we will
give a direct construction based on the same ideas. The construction
is also similar to Kan completion, as referred to in
\cite{bchcubicalsets}.  This approach has three main advantages.

Firstly, this allows us to ensure that the proof holds in a
constructive setting (such as the set theory $\czf$), whereas it is
not clear whether the proof in \cite{garnersmallobject} is valid
constructively.

Secondly, it allows us to see explicitly the objects involved in the
construction, which may be useful if one wanted to use the ideas here
in computer implementations.

Thirdly, we will see that the proof uses only countable colimits, not
requiring the full cocompleteness of $\sub$. Although $\sub$ is
cocomplete, this property should make it easier to apply the proof in
other contexts. For example, if one were to define ``01-substitution
assemblies'' by analogy with Stekelenburg's recent work on simplicial
assemblies in \cite{stekelenburgsimpass} this may be useful.

We will also suggest how to construct stable path objects. In a future
paper, the author will use this to give a constructive model of
homotopy type theory including the computational rule for identity,
which is absent in \cite{bchcubicalsets}.

\subsection{Cubical Sets and 01-Substitution Sets}
\nocite{pittsnomsets}
\nocite{pittssub}

Cubical sets were developed by Bezem, Coquand and Huber in
\cite{bchcubicalsets} (and described in more detail in
\cite{huberthesis}) as a constructive model of homotopy type theory,
based on the simplicial set model due to Voevodsky.

Pitts showed in \cite{pittssub} and \cite{pittsnompcs}, following
earlier work by Staton, that the category of cubical sets is
equivalent to a category based on nominal sets, called
$01$-substitution sets. In this paper we will work over this category,
$\sub$ of $01$-substitution sets, using the notation and definitions
of open box and Kan fibration that appear in \cite{pittssub}. For a
good introduction to nominal sets, on which 01-substitution sets are
based, see \cite{pittsnomsets}.

We recall the following definitions from \cite{pittsnomsets}. Let
$\names$ be a set (which we will refer to as the set of
\emph{names}). Write $\perma$ for the group of finite permutations
(that is, permutations $\pi$ such that $\pi(a) = a$ for all but
finitely many $a \in \names$). Recall that a $\perma$-set is a set
$X$, together with an action of $\perma$ on $X$ (or equivalently a
presheaf over $\perma$ when viewed as a one object category in the
usual way).

\begin{definition}[Pitts, Gabbay]
  \begin{enumerate}
  \item Let $X$ be a $\perma$-set (writing $\cdot$ for
    the action) and let $x \in X$. We say $A \subseteq \names$ is a
    \emph{support} for $x$ if whenever $\pi(a) = a$ for all $a \in A$,
    also $\pi \cdot x = x$.
  \item Let $X$ be a $\perma$-set and $x \in X$. We say $x$ is
    \emph{equivariant} if $\pi \cdot x = x$ for all $\pi \in \perma$,
    (or equivalently if $\emptyset$ is a support for $x$).
  \item Let $X$ and $Y$ be $\perma$-sets. A function $f: X \rightarrow
    Y$ is \emph{equivariant} if it is a morphism in the category of
    $\perma$ sets, or equivalently if it is equivariant as an element
    of the exponential $Y^X$ in the category of $\perma$ sets, which
    is described explicitly as the set of functions $X$ to $Y$ with
    action given by conjugation.
  \item A \emph{nominal set} is a $\perma$-set $X$,
    such that for every $x \in X$, there exists a finite set $A
    \subseteq \names$ such that $A$ is a support for $x$.
  \item Let $X$ and $Y$ be nominal sets and let $x \in X$ and $y \in
    Y$. We say $x$ is \emph{fresh} for $y$ and write $x \# y$ if there
    exist finite sets $A, B \subseteq \names$ such that $A$ is a
    support for $x$, $B$ is a support for $y$ and $A \cap B =
    \emptyset$.
  \item Let $X$ and $Y$ be nominal sets. The \emph{separated product},
    $X \ast Y$, of $X$ and $Y$ is the nominal set with elements
    $(x, y)$ where $x \in X$, $y \in Y$ and $x \# y$ (with the action
    defined componentwise).
  \item Let $X$ be a nominal set. We define an equivalence relation
    $\sim$ on $\names \times X$, referred to as
    \emph{$\alpha$-equivalence} as follows. $(a, x) \sim (a', x')$ if
    and only there exists $a''$ fresh for $a, x, a', x'$ such that
    $(a'' \; a) x = (a'' \; a') x'$. The quotient
    $\names \times X / \sim$ can be viewed in a natural way as a
    nominal set $[\names]X$, referred to as the \emph{name
      abstraction} on $X$. We write $\langle a \rangle x$ for the
    equivalence class of $(a, x)$.
  \end{enumerate}
\end{definition}

The lemma below is a variant of standard results in nominal sets as in
\cite[Chapter 4]{pittsnomsets}. It allows us to easily construct
morphisms $F: W \rightarrow Y$ when $W \subseteq [\names]X$ for some
nominal sets $X$ and $Y$ by the following heuristic. We assume we have
been given some finite list of parameters. We then define $F(\langle a
\rangle x)$ when $a$ is fresh for the parameters, to get a partial
function on the preimage of $W$ in $\names \times X$. This can be done
without worrying whether or not $F$ respects $\alpha$-equivalence (a
priori). We then check that $a$ is fresh for $F(\langle a \rangle
x)$. We then check or note by inspection that $F$ is equivariant
modulo the parameters (ie the list of parameters is a support for
$F$).  We then apply the lemma to get a well defined morphism, that in
particular necessarily respects $\alpha$-equivalence by well
definedness.

\begin{lemma}
  \label{lem:nameabsext}
  Let $Z$ be a subobject of $\names \times X$ (ie an equivariant
  subset) and write $Z'$ for the image of $Z$ under the projection
  $\names \times X \rightarrow [\names]X$. Let $Y$ be another nominal
  set and $F : Z \rightharpoondown Y$ a partial function with finite
  support (under the action given by conjugation)
  such that for any $a \# F$, if $(a, x) \in Z$, then $F(a, x)
  \downarrow$ and $a \# F(a,x)$. Then $F$ extends uniquely to a function
  $\bar{F}: Z' \rightarrow Y$. Furthermore, $\bar{F}$ is constructed
  equivariantly, in the following sense: for any $\pi \in \perma$,
  $\overline{\pi \cdot F} = \pi \cdot \overline{F}$.
\end{lemma}

\begin{proof}
  This is a slight generalisation of \cite[Theorem 4.15]{pittsnomsets}
  and the same proof works here.
\end{proof}

We will also use the following corollary.

\begin{corollary}
  \label{cor:nameabsextexp}
  Let $Z$ and $Z'$ be the same as in lemma \ref{lem:nameabsext}. Let
  $W$ be another nominal set, and let $F: W \times Z \rightharpoondown
  Y$ be a partial function with finite support $A \subseteq \names$
  such that for any $(a, x) \in Z$ and $w \in W$, if $a \# F$ and $a
  \# w$ then $F(a, x) \downarrow$ and $a \# F(w, (a, x))$. Then $F$
  extends uniquely to a map $\bar{F}: W \times Z' \rightarrow Y$ and
  $A$ is a support for $\bar{F}$.
\end{corollary}

\begin{proof}
  For each $w \in W$, we have a partial function $F(w, -): Z
  \rightharpoondown Y$. Furthermore, if $B$ is a finite support for
  $w$, then $A \cup B$ is a finite support for $F(w, -)$. Applying
  lemma \ref{lem:nameabsext} gives us a function $\bar{F}_w : Z'
  \rightarrow Y$. We then define $F : W \times Z' \rightarrow Y$ by
  $\bar{F}(w, z) := \bar{F}_w(z)$.
\end{proof}

We now recall from \cite{pittssub} and \cite{pittsnompcs} Pitts'
definition of $01$-substitution sets and his translation of the
Bezem-Coquand-Huber definitions of open box and Kan fibration.

\begin{definition}[Pitts]
  Let $X$ be a nominal set. A $01$-substitution operation on $X$ is a
  morphism $s:X \times \names \times 2 \rightarrow X$ in nominal sets,
  satisfying the following axioms. We write $x(a := i)$ for $s(x, a,
  i)$. For any $x \in X$, $a, a' \in \names$ and $i, i' \in 2$,
  \begin{enumerate}
  \item $a \,\#\, x(a := i)$
  \item if $a \,\#\, x$, then $x(a := i) \,=\, x$
  \item if $a \neq a'$, then $x(a := i)(a' := i') \,=\, x(a' := i')(a :=
    i)$ 
  \end{enumerate}

  We say that a nominal set equipped with a $01$-substitution
  operation is a \emph{$01$-substitution set}. $01$-substitution sets
  form a category $\sub$, where morphisms are morphisms in nominal
  sets that also preserve the $01$-substitution operation.
\end{definition}

Note that the name abstraction $[\names]X$ of a $01$-substitution set
$X$ can be viewed itself as a $01$-substitution set in a canonical
way.

We now give the definitions of open box and Kan fibration.

\begin{definition}[Pitts]
  \begin{enumerate}
  \item Let $A$ be a finite subset of $\names$ with $a \in A$ and let
    $f : X \rightarrow Y$ be a morphism in $\sub$. A \emph{$1$-open
      $(A, a)$-box over $f$} is a pair $(u, y)$ with $y \in Y$ and
    $u : (A \times 2) \setminus \{(a, 1)\} \rightarrow X$ satisfying
    the following. For all $(b, i), (b', i') \in (A \times 2)
    \setminus \{(a, 1)\}$,
    \begin{enumerate}
    \item $b \,\#\, u(b, i)$
    \item $u(b, i)(b' := i') \,=\, u(b', i')(b := i)$
    \item $f(u(b, i)) \,=\, y(b := i)$
    \end{enumerate}
  \item We also define \emph{$0$-open (A, a)-box over $f$} by simply
    replacing $1$ with $0$ in the above definition.
  \item Let $(u, y)$ be a $1$-open $(A, a)$-box over $f$. A
    \emph{filler} for $(u, y)$ is $x \in X$ such that 
    \begin{enumerate}
    \item for all $(b, i)
      \in (A \times 2) \setminus \{(a, 1)\}$, $x(b := i) = u(b, i)$
    \item $f(x) = y$
    \end{enumerate}
  \item We similarly define fillers for $0$-open boxes.
  \item Let $f : X \rightarrow Y$. A \emph{Kan filling operator for
      $1$-open boxes} is for each $1$-open box $(u, y)$, a choice of
    filler, $\uparrow (u, y)$ satisfying the following conditions
    (which we refer to as \emph{uniformity conditions}).
    \begin{enumerate}
    \item For each finite permutation $\pi$, $\uparrow (\pi (u, y)) =
      \pi (\uparrow (u, y))$.
    \item Whenever $(u, y)$ is a $1$-open $(A, a)$-box with $c \# A$
      and $i \in 2$, we have $\uparrow ((u, y)(c := i)) = (\uparrow
      (u, y))(c := i)$.
    \end{enumerate}
  \item We similarly define the notion of Kan filling operator
    $\downarrow (u, y)$ for $0$-open boxes $(u, y)$.
  \item We say $f$ is a \emph{Kan fibration} (or simply
    \emph{fibration}) if it admits both Kan filling operators.
  \end{enumerate}
\end{definition}

\begin{remark}
  Note that the definition of open box includes the case where $A =
  \{a\}$, and so the ``box'' only consists of one point $x \in X$ and
  an element of $Y$. This is an important special case, and
  essentially says that any path in $Y$ with endpoint $f(x)$ can be
  lifted to a path in $X$ with endpoint $x$.
\end{remark}

\begin{remark}
  Note that as a special case of the uniformity conditions we have
  that whenever $\pi$ is a permutation of a finite set $A \subseteq
  \names$ that fixes $a \in A$ and $u$ is a $1$-open $(A, a)$-box we
  have that if $\pi \cdot u = u$ then $\pi (\uparrow u) = \uparrow
  u$. Hence the uniformity conditions provide conditions on each
  filler individually, not just conditions on how the fillers of
  different open boxes relate to each other. In classical logic with
  the axiom of choice, the existence of uniform Kan filling operators
  is equivalent to the existence of fillers satisfying this
  ``symmetry preserving'' condition, which is stronger than just requiring
  fillers without the condition.
\end{remark}

\subsubsection{A Note on Nominal Sets in a Constructive Setting}

As in \cite{bchcubicalsets} we require the assumption that the set
$\names$ of names has decidable equality. Note in particular that for
sets with decidable equality, finite and finitely enumerable subsets
coincide (ie if a subset of $\names$ is the image of a function from a
natural number to $\names$, then it is the image of an injection from
a natural number to $\names$) and these are decidable subsets (ie if
$A \subseteq \names$ is finite then every element of $\names$ belongs
to $A$ or does not). We also assume that $\names$ is infinite in the
strong sense that for every finite $A \subseteq \names$ there exists
$a \in \names$ such that $a \notin A$.

In \cite{pittsnomsets} heavy use is made of the existence of least
finite support. Constructively this can't be guaranteed to exist; the
proof of theorem \ref{thm:nonarlm} provides an example where the
freshness relation is not decidable and hence there cannot be a least
support that is also a finite set. In fact one can show that in
general one cannot even prove constructively that any least support
exists, finite or otherwise.  In practice, however, most of the uses
of least finite support in \cite{pittsnomsets} can be viewed as a
``notational device'' to make definitions and statements of theorems
more concise, and these can be rephrased to work in a constructive
setting (see the work by Choudhury in \cite{choudhuryconnom}).

In some places we will for convenience assume the axiom of dependent
choice, but this will not be required for the main results.

\subsubsection{The Nerve of a Complete Metric Space}

To help us give some non trivial examples of objects of $\sub$ later,
we will use the notion from homotopical algebra of \emph{nerve}. This
gives a way to construct $01$-substitution sets from topological
spaces.

For this subsection suppose that $\names = \nat$. Also, for
convenience we will assume the axiom of dependent choice, since it is
often implicitly assumed in Bishop style analysis that we use here.

Note that $\rint^\nat$ has a canonical metric given by
the product metric (see \cite[Chapter 4, definition
1.7]{bishopbridges}), defined as follows for $r, r' \in \rint^\nat$.
\begin{equation}
  \label{eq:5}
  d(r, r') := \sum_{n \in \nat} \, 2^{-n} |r(n) - r'(n)|
\end{equation}

Let $X$ be a metric space. Then we define an action on uniformly
continuous functions $F : \rint^\nat \rightarrow X$ as follows. For
$\pi \in \perma$ and $r \in \rint^\nat$,
\begin{equation}
  \label{eq:7}
  (\pi \cdot F)(r) := F(r \circ \pi)
\end{equation}

We then define the \emph{nerve} of $X$, $N(X)$ to be the subobject of
the $\perma$ set defined above consisting of elements that have finite
support. Note that this is a nominal set by definition.

\begin{lemma}
  \label{lem:nervsupp}
  Let $A$ be a finite subset of $\nat$. Then for $F \in N(X)$, $A$ is
  a support for $F$ if and only if for every $r, r' \in \rint^\nat$ if
  $r|_A = r'|_A$ then $F(r) = F(r')$.
\end{lemma}

\begin{proof}
  Assume first that $A$ is a support for $F$. We need to show that for
  $r, r' \in \rint^\nat$ with $r|_A = r'|_A$, we have $F(r) =
  F(r')$. Note that it suffices to show that for all $\epsilon > 0$,
  $d(F(r), F(r')) < \epsilon$ (for instance by
  \cite[Chapter 2, Lemma 2.18]{bishopbridges}).

  For any $\epsilon > 0$, there exists by uniform continuity of $F$,
  $\delta > 0$ such that for all $x, y \in \rint^\nat$ if $d(x, y) <
  \delta$ then $d(F(x), F(y)) < \frac{\epsilon}{2}$. Let $N$ be such
  that $2^{-N} < \delta$ and such that $n < N$ for all $n \in
  A$. Define $B$ to be $\{0, \ldots, N \} \setminus A$ and let
  $\{b_1,\ldots,b_k\}$ be an enumeration of $B$, with $b_i \neq b_j$
  for $i \neq j$.

  Then define $\pi$ to be the composition of transpositions
  \begin{equation}
    \pi := (b_1 \quad N + 1)(b_2 \quad N + 2) \ldots (b_k \quad N + k)
  \end{equation}
  and define $r'' \in \rint^\nat$ by
  \begin{equation}
    \label{eq:9}
    r''(n) :=
    \begin{cases}
      r(n) & n \leq N \text{ or } n > N + k \\
      r'(b_i) & n = N + i \text{ for some } 1 \leq i \leq k
    \end{cases}
  \end{equation}
  Then note that $r$ and $r''$ agree on the set $\{0,\ldots, N\}$, so
  by the definition of the metric on $\rint^\nat$ we have that $d(r,
  r'') \leq 2^{-N} < \delta$. Similarly, we have $d(r', r'' \circ \pi)
  \leq 2^{-N} < \delta$. Further, since $\pi$ fixes $A$, which is a
  support for $F$, we have
  \begin{align}
    F(r'' \circ \pi) &= (\pi \cdot F)(r'') \\
    &= F(r'')
  \end{align}
  Therefore we have
  \begin{align}
    d(F(r), F(r')) &\leq d(F(r), F(r'')) + d(F(r'' \circ \pi), F(r'))
    \\
    &< \frac{\epsilon}{2} + \frac{\epsilon}{2} \\
    &= \epsilon
  \end{align}
  as required.

  Finally, the converse is easy to show by noting that if
  $\pi$ fixes $A$ then $r|_A = (r \circ \pi)|_A$ for all $r \in
  \rint^\nat$.
\end{proof}

We now define a substitution operation on $N(X)$ as follows. For $F
\in N(X)$, $r \in \rint^\nat$, $a \in \nat$, $i \in \{0, 1\}$ and $b
\in \nat$, define
\begin{equation}
  \label{eq:15}
  F(a := i)(r)(b) :=
  \begin{cases}
    r(b) & b \neq a \\
    i & b = a
  \end{cases}
\end{equation}
One can use lemma \ref{lem:nervsupp} to easily show that this
satisfies the axioms for $01$-substitution set.

Also note that we can use lemma \ref{lem:nervsupp} to prove the
following lemma, that we will use later.

\begin{lemma}
  \label{lem:ctsfunctionsnerve}
  \begin{enumerate}
  \item   For any $a \in \nat$ there is a correspondence between $F
    \in N(X)$ such that $\{a\}$ is a support for $F$ and uniformly
    continuous functions $\rint \rightarrow X$.
  \item Such an $F$ corresponds to a constant function $\rint
    \rightarrow X$ if and only if $\emptyset$ is a support for $F$.
  \end{enumerate}
\end{lemma}

\begin{proposition}
  \label{prop:nervefibrant}
  Let $X$ be a complete metric space. Then the unique map $N(X)
  \rightarrow 1$ is a fibration. (We say that $N(X)$ is
  \emph{fibrant}.)
\end{proposition}

\begin{proof}
  We are given an open box $u : A \times 2 \setminus (a, 1)$ for some
  finite set $A \subseteq \nat$ and $a \in A$ and we need to construct
  a filler in such a way that this can be done uniformly.

  For each $(b, i) \in A \times 2$, define $X_{b, i} \subseteq
  \rint^A$ by
  \begin{equation}
    \label{eq:16}
    U_{b, i} := \{ r \in \rint^A \;|\; r(b) = i \}
  \end{equation}
  Note that we can easily view $u$ as a uniformly continuous function
  \begin{equation}
    \label{eq:17}
    \bar{u} : \bigcup_{(b, i) \in A \times 2 \setminus (a, 1)} U_{b,
      i} \quad \times \quad \rint^{\nat \setminus A} \quad \rightarrow
    \quad X
  \end{equation}

  Note that one can piecewise linearly define a uniformly continuous
  retraction, $\tau$ from a dense subset $D$ of $\rint^A$ to
  $\bigcup_{(b, i) \in A \times 2 \setminus (a, 1)} U_{b, i}$. We then
  define a function
  $\tau' : D \times \rint^{\nat \setminus A} \rightarrow \bigcup_{(b, i) \in A
    \;\times\; 2 \setminus (a, 1)} U_{b, i} \times \rint^{\nat \setminus
    A}$
  by ``passing through $\rint^{\nat \setminus A}$ unchanged.''  Then
  the composition $\bar{u} \circ \tau'$ is a uniform continuous
  function from a dense subset of $\rint^\nat$ to a complete metric
  space. Under these conditions one can extend $\bar{u} \circ \tau$ to
  a uniformly continuous function $\rint^\nat \rightarrow X$ (for
  example, see \cite[Chapter 4, Lemma 3.7]{bishopbridges}).

  To ensure the uniformity conditions for the Kan filling operator,
  note that we can construct piecewise linearly for each finite
  $A \subseteq \nat$ and $a \in A$, dense subsets
  $U_A \subseteq \rint^A$ and retractions
  $\tau_A : D_A \rightarrow \bigcup_{(b, i) \in A \times 2 \setminus
    (a, 1)} U_{b, i}$ with the following symmetry property. Let $\pi
  \in \perma$ and write $\bar{\pi}$ for the function $\rint^{\pi A}
  \rightarrow \rint^A$ induced by composition. Then $D_{\pi A} =
  \bar{\pi}^{-1}(D_A)$ and $\tau_{\pi A} = \bar{\pi}^{-1} \circ \tau_A
  \circ \bar{\pi}$. If we use retractions with these
  symmetry conditions in the proof above we ensure the uniformity
  conditions are satisfied.
\end{proof}

\begin{remark}
  When we try to define a function on $\rint^\nat$ piecewise linearly,
  the best we can do constructively in general is to define the
  function on a dense subset. This is why we work with dense subsets
  of $\rint^\nat$ in the proof of proposition \ref{prop:nervefibrant},
  and why we need the extra assumption of completeness of $X$. See
  \cite{palmgrenfh}, where Palmgren discusses a similar issue for a
  related property called the \emph{path joining property}. In fact
  the path joining property follows from the fibrancy of the nerve of
  a metric space, so by \cite[Proposition 2.3]{palmgrenfh} we
  cannot show the nerve of $\{x \in [-1, 1] \;|\; x \leq 0 \vee x \geq
  0 \}$ is fibrant without assuming LLPO (and also we cannot show
  that this space is complete without assuming LLPO).

  As in \cite{palmgrenfh}, another approach that promises to be better
  behaved for more general results is to use formal topologies instead
  of metric spaces. In this paper all the examples we use will be
  complete metric spaces.
\end{remark}

\subsection{Algebraic Weak Factorisation Systems}
\nocite{riehlams}
\nocite{grandistholennwfs}

Weak factorisation systems are widely used in homotopical
algebra, and are defined as follows.

\begin{enumerate}
\item Let $\catc$ be a category, and let $i : U \rightarrow V$ and $f
  : X \rightarrow Y$ be morphisms in $\catc$. We say $i$ has the
  \emph{left lifting property} with respect to $f$, and $f$ has the
  \emph{right lifting property} with respect to $i$ and write $i
  \pitchfork f$ to mean that for every commutative square of the
  following form,
  \begin{equation}
    \label{eq:19}
    \begin{gathered}
      \xymatrix{ U \ar[r] \ar[d]_i & X \ar[d]_f \\
        V \ar[r] & Y }
    \end{gathered}
  \end{equation}
  there exists a diagonal map $j : V \rightarrow X$ making two
  commutative triangles.
\item Let $\mathcal{M}$ be a class of morphisms. We define the
  following classes.
  \begin{align}
    M^\pitchfork &:= \{ f \;|\; (\exists i \in \mathcal{M})\,i
                   \pitchfork f \} \\
    {}^\pitchfork M &:= \{ i \;|\; (\exists f \in \mathcal{M})\,i
                   \pitchfork f \}
  \end{align}
\item For classes $\mathcal{M}$ and $\mathcal{N}$ we write
  $\mathcal{M} \pitchfork \mathcal{N}$ to mean that for all $i \in
  \mathcal{M}$ and for all $f \in \mathcal{N}$ we have $i \pitchfork
  f$.
\item A \emph{weak factorisation system} is two classes of maps
  $\mathcal{L}$ and $\mathcal{R}$ such that $\mathcal{L} =
  {}^\pitchfork \mathcal{R}$, $\mathcal{R} = \mathcal{L}^\pitchfork$
  and every morphism in $\catc$ factors as a morphism in $\mathcal{L}$
  followed by a morphism in $\mathcal{R}$.
\end{enumerate}

Algebraic weak factorisation systems (originally called
natural weak factorisation systems) are a variation developed by
Grandis and Tholen in \cite{grandistholennwfs}. Garner showed in
\cite{garnersmallobject} that a form of the small object argument can
be used to construct awfs's from a diagram of left maps. However, the
proof uses transfinite arguments that may be problematic
constructively.  Riehl used awfs's in \cite{riehlams} as the main
ingredient in the theory of algebraic model structures; the same
paper contains a comprehensive introduction to awfs's.

We recall the definition of awfs below. We write $\btwo$ and $\bthree$
for the categories given by linear orderings with $2$ elements and $3$
elements respectively.

\begin{definition}
  Let $\catc$ be a category. Note that there is a canonical functor
  $\catc^\bthree \rightarrow \catc^\btwo$ given by composition. A
  \emph{functorial factorisation} on $\catc$ is a functor $\catc^\btwo
  \rightarrow \catc^\bthree$ that is a section of the composition
  functor.
\end{definition}

Throughout this paper we will write out functorial factorisations as
three separate components: a functor $K : \catc^\btwo \rightarrow
\catc$ together with maps $\lambda_f$ and $\rho_f$ for each morphism
$f : X \rightarrow Y$ in $\catc$, as in the diagram
below:
\begin{equation}
  \begin{gathered}
    \label{eq:2}
    \xymatrix{ X \ar[rr]^f \ar[dr]_{\lambda_f} & & Y \\
      & K f \ar[ur]_{\rho_f} & }
  \end{gathered}
\end{equation}

For each functorial factorisation we may define a copointed
endofunctor $L : \catc^\btwo \rightarrow \catc^\btwo$, whose action on
objects is given by sending $f$ to $\lambda_f$, and a pointed
endofunctor $R : \catc^\btwo \rightarrow \catc^\btwo$, whose action on
objects is given by sending $f$ to $\rho_f$. For full details see
\cite{grandistholennwfs}, \cite{garnersmallobject} or
\cite{riehlams}.

\begin{definition}[Grandis, Tholen]
  Let $\catc$ be a category. An \emph{algebraic weak factorisation
    system on $\catc$} consists of a functorial factorisation
  $K, \lambda, \rho$, together with natural transformations
  $\Sigma : L \rightarrow L^2$ and $\Pi : R^2 \rightarrow R$ such that
  $L$ together with $\Sigma$ form a comonad, $\monadf{L}$ (with
  $\Sigma$ the comultiplication map), and $\Pi$ together with
  $R$ form a monad $\monadf{R}$.
\end{definition}

Note in particular that the multiplication on $R$ gives us an
$R$-algebra structure on $\rho_f$ for any $f$ (as is the case for any
monad), and dually comultiplication gives $\lambda_f$ the structure of
an $L$-coalgebra for every $f$.

\section{Construction of Functorial Factorisation}

\subsection{Construction of Factorisation}
\label{sec:konobjs}

We now define the awfs on $\sub$. The basic idea is the same as
Garner's small object argument, as in \cite{garnersmallobject}, but is
somewhat simpler here than in general. This can also be seen as a
generalisation of Kan completion, as defined by Huber in
\cite[Section 3.5]{huberthesis}.

We first define the functor $K : \sub^\btwo \rightarrow \sub$ that
will provide the objects of the functorial factorisation $\sub^\btwo
\rightarrow \sub^{\mathbf{3}}$. The basic intuition here is that we
know the map $\rho : K f \rightarrow Y$ should be a Kan
fibration. Hence we freely add fillers for open boxes to $K f$ to
ensure this is the case. We use the set of open boxes itself to do
this, with the idea that each open box is its own filler. In order for
$K f$ to be an object in $\sub$ we need to ensure that if $(u, y)$ is
a $1$-open $(A, a)$-box that has been added to $K f$, then the
substitution $(u, y)(a := 1)$ is well defined. For this, we add
another component $K^+ f$ which is a subset of $[\names] K f$,
corresponding to what Bezem, Coquand and Huber refer to as \emph{Kan
  composition} in \cite{bchcubicalsets}. The remaining substitutions
are already determined by the conditions on Kan filler operations,
including the uniformity conditions so we don't need to add anything
more regarding $1$ open boxes. We then do the same thing for $0$-open
boxes and iterate $\omega$ times.

Fix $f : X \rightarrow Y$.  We will define $K f$ in $\omega$
stages. For each $n \in \omega$, we will define inductively $K_n f \in
\sub$. We will simultaneously define $\rho_n : K_n f \rightarrow Y$
and $\lambda_n : K_n f \rightarrow K_{n + 1} f$. We will also ensure
that $\rho_{n+1} \circ \lambda_n = \rho_n$.

If we have already defined
$K_m f$ for $m < n$ then we define $\colim_{m < n} K_m f$ to be the
colimit over $K_m f$ for $m < n$ together with the maps $\lambda_m$.
We now define $K_n f$, assuming that $K_m f$ has already been defined
for $m < n$. We first define nominal sets $K_n^\uparrow f$, $K_n^+ f$,
$K_n^\downarrow f$ and $K_n^- f$.

\begin{enumerate}
\item Define $K_n^\uparrow f$ to be pairs $(u, y)$ where $u$ is a $1$-open box
in $\colim_{m < n} K_m f$ over $y$. The action of permutations on
$K_n^\uparrow f$ is defined componentwise.
\item Define $K_n^+ f$ to be the subset of $[\names] K_n^\uparrow f$
  consisting of $\langle a \rangle (u, y)$ where $u$ is a $1$-open $A,
  a$-box in $\colim_{m < n} K_m f$ over $y$. More formally, we define
  an equivariant subset of $\names \times K_n^\uparrow f$ as the set
  of pairs $(a, (u, y))$ where $u$ is an $A, a$-box for some finite
  $A$ with $a \in A$, and then define $K_n^+ f$ to be image of this
  subset under the projection $\names \times K_n^\uparrow f
  \rightarrow [\names]K_n^\uparrow f$.
\item Define $K_n^\downarrow f$ analogously to $K_n^\uparrow f$ but
  for $0$-open boxes.
\item Define $K_n^- f$ analogously to $K_n^+ f$ but for $0$-open
  boxes.
\end{enumerate}

We now define $K_n f$ to be coproduct $X \amalg K_n^\uparrow f
\amalg K_n^+ f \amalg K_n^\downarrow f \amalg K_n^- f$ in
nominal sets. Since we have not defined a substitution structure on
the individual components, it would not make sense to use the
coproduct in $\sub$. We will show below how to define a substitution
structure on $K_n f$.

Note that there is a
natural injection of $K_{n - 1}^\uparrow f$ into $K_n^\uparrow f$, and
similarly for the other components of the disjoint union. Hence we can
define $\lambda_{n - 1}$ componentwise. We also define $\rho_n$
componentwise, as follows.
\begin{enumerate}
\item If $x \in X$, we define $\rho_n(x) = f(x)$.
\item If $(u, y) \in K_n^\uparrow f$ is a $1$-open $A, a$-box, we
  define $\rho_n(u, y) = y$.
\item If $\langle a \rangle (u, y) \in K^+_n f$, we define
  $\rho_n(\langle a \rangle (u, y)) := y (a := 1)$. Formally, to show
  this is well defined, first use the above description to get a
  partial function on pairs $(a, (u, y))$ where $(u, y)$ is a $1$-open
  $A, a$ box and apply lemma \ref{lem:nameabsext}. Note that $a$ is
  fresh for $y(a := 1)$ by the axioms for $01$-substitution sets.
\item For $K_n^\downarrow f$, we define $\rho$ analogously to
  $K_n^\uparrow f$.
\item For $K_n^- f$, we define $\rho$ analogously to
  $K_n^+ f$.
\end{enumerate}

Since $K_n f$ was defined as a coproduct in nominal sets, we already
have an action of $\perma$ on $K_n f$ (which is just defined
componentwise), but we still need to define the action of
substitutions. If $x$ is in the ``copy'' of $X$ in $K_n f$, then
define $x (a := i)$ to be the same as in $X$ itself. If $(u, y)$ is an
element of $K_n^\uparrow f$, with $u$ an $A, a$ box then define
$(u,y)(a' := i)$ by cases. In the below, we write $\lambda_{m < n}$ to
mean the natural injection of $\colim_{m < n} K_m f$ into $K_n f$.
\begin{equation}
(u,y)(a' := i) :=
\begin{cases}
  \lambda_{m < n} (u(a',i)) 
  & \text{if } (a',i) \in A \times 2 \setminus (a,1) \\
  \langle a \rangle (u, y) \in K_n^+ f
  & \text{if } a' = a \text{ and } i = 1 \\
  (u (a' := i), y (a' := i)) \in K_n^\uparrow f
  & \text{otherwise}
\end{cases}
\end{equation}
Substitution for $K_n^\downarrow f$ is defined similarly.

If $\langle a \rangle (u, y)$ is an element of $K_n^+ f$, with $u$ a
$1$-open $A,a$-box, then we define $\langle a \rangle (u, y) (a' :=
i)$ as follows. Note that by applying corollary
\ref{cor:nameabsextexp} we may assume without loss of generality that
$a' \neq a$.
\begin{equation}
\langle a \rangle (u, y) (a' := i) :=
\begin{cases}
  \lambda_{m < n} (u(a', i) (a := 1)) & \text{if } a' \in A \\
  \langle a \rangle (u (a' := i), y (a' := i)) & \text{otherwise}
\end{cases}
\end{equation}
To check that we do satisfy the conditions of corollary
\ref{cor:nameabsextexp}, we note that in both cases $a$ is fresh for
$\langle a \rangle (u, y) (a' := i)$.  Substitution for $K_n^- f$ is
defined similarly.

To show that $K_n f$ is a 01-substitution set, we need to check the
following axioms.
\begin{gather}
  \pi (x (a := i)) = \pi x (\pi a := i) \\
  a \# x (a := i) \label{subfresh} \\
  a \# x \Rightarrow x (a := i) = x \\
  a \# a' \Rightarrow x (a := i) (a' := i') = x (a' := i') (a := i)
  \label{subcomm}
\end{gather}
These can be checked by induction on $n$, splitting into cases
depending on which disjoint component of $K_n f$ $x$ lies in. For
\eqref{subfresh} and \eqref{subcomm} it is necessary to use
the freshness and adjacency conditions respectively in the definition
of open box.

We now define $K f$ to be the colimit over all $K_n f$ (together with
the injections $\lambda_n$). We also have an obvious morphism
$\lambda_f : X \rightarrow K f$ given by inclusion $X$ into $K_0 f$,
and a morphism $\rho_f : K f \rightarrow Y$ using the $\rho_n$.

\subsection{Rank is Well Defined}

Note that for every $n$ there is a canonical map $\lambda_{n < \omega}
: \colim_{m < n} K_n f \rightarrow K f$. In the following lemma we show
that for every $x \in K f$ there is a least $n$ such that $x$ appears
in the image of such a map. Note that constructively, we are forced to
find this $n$ explicitly, rather than appealing to the well ordering of
$\omega$.

\begin{lemma}
  \label{lem:rank}
  For every $x \in K f$, there is a least $n$ such that there is $x'
  \in K_n f$ such that $\lambda_{n < \omega} (x') = x$. Furthermore, $x'$ is
  uniquely determined.
\end{lemma}

\begin{proof}
  We will show by induction that for every element $x$ of $\coprod_{m <
    \omega} K_m f$, the statement holds for the equivalence class
  $[x]$.

  Suppose that we have proved the statement for $[y]$ for all $y \in
  \coprod_{m' < m} K_{m'} f$. Now given $x \in K_m f$, note that $x$ must
  belong to one of the five components in the disjoint union forming
  $K_m f$. If $x \in X$, then note that we can take $n$ to be $0$ and
  $x'$ to be $x$.

  If $x$ is of the form $(u, y)$ where $u$ is a $1$-open $A, a$-box,
  then note that for every $(a',i) \in A \times 2 \setminus (a,1)$, we
  have that $u(a',i) \in \colim_{m' < m} K_{m'} f$. Given $(a', i)$ in
  $A \times 2 \setminus (a,1)$, let $w$ be a representative of the
  equivalence class $u(a',i)$. Note that $w$ is an element of $K_{m'}
  f$ for some $m' < m$. Hence by induction there is $n$ least such
  that there is $z$ in $K_{n} f$ with $z$ equivalent to $w$. Note that
  $z$ and $n$ are independent of the choice of equivalence class
  representative $w$ and so they depend only on $(a',i)$. We write
  them now as $z_{a',i}$ and $n_{a',i}$. Since $\{n_{a'i} \;|\; (a',i)
  \in A \times 2 \setminus (a,1) \}$ is a finite subset of $\omega$ it
  must have a greatest element, $N$. Then we can take $n$ in the
  statement of the lemma to be $N + 1$. We take $x'$ to be given by
  $(u, y)$ where $u$ is the open box with $u(a',i) := z_{a'i}$. Note
  in particular that since each $\lambda_{n < m}$ is an injection, so
  is each map $\lambda_{n < \omega}$ and so $u$ satisfies the
  adjacency conditions and hence really is an open box.

  If $x$ belongs to one of the remaining 3 components, then the proof
  is similar to that of $K_n^\uparrow f$.
\end{proof}

\begin{definition}
  \label{def:rank}
  Given $x \in K f$, we refer to the $n$ in lemma \ref{lem:rank} as
  the \emph{rank} of $x$.
\end{definition}

\begin{remark}
  We will see later that the fact that rank is a well defined natural
  number implies that all acyclic cofibrations have decidable image
  (lemma \ref{lem:cofibimdec}). Constructively, this is not the case
  in general for cofibrantly generated awfs's.
\end{remark}

\subsection{Functoriality of $K$}

In section \ref{sec:konobjs} we defined $K$ on the objects of
$\sub^\btwo$. We now complete the construction of the functor by
defining the action of $K$ on the morphisms of $\sub^\btwo$.

Given $f: X \rightarrow Y$ and $g : U \rightarrow V$ objects of
$\sub^\btwo$, recall that a morphism from $f$ to $g$ is a commutative
square
\begin{equation}
\begin{gathered}
\xymatrix{ X \ar[r]^h \ar[d]_f & U \ar[d]^g \\
  Y \ar[r]_k & V
}
\end{gathered}
\end{equation}

We also need to check that the following diagram commutes (in order to
ensure that the overall factorisation $\sub^\btwo \rightarrow
\sub^{\mathbf{3}}$ is a functor).
\begin{equation}
\label{eq:khknat}
\begin{gathered}
\xymatrix{ X \ar[r]^h \ar[d]_{\lambda_f} & U \ar[d]^{\lambda_g} \\
  K f \ar[r]^{K(h,k)} \ar[d]_{\rho_f} & K g \ar[d]^{\rho_g} \\
  Y \ar[r]^k & V
}
\end{gathered}
\end{equation}
We will simultaneously check this while defining $K(h,k)$. One can
also check at the same time that $K(h,k)$ is equivariant and preserves
substitution, ie that it actually is a morphism in $\sub$.

We define $K(h,k) : K f \rightarrow K g$ by induction on rank. If $x
\in K f$ has rank $0$, then it is an element of $X$. We define $K(h,k)
(x)$ to be $h(x)$. Note that this precisely ensures that the upper
square of \eqref{eq:khknat} commutes.

If $x$ is of the form $(u, y)$ where $u$ is a 1-open $A,a$-box, then
note that for each $(a', i) \in A\times 2 \setminus (a, 1)$ we have
that $u(a',i)$ is of strictly lower rank than $x$, and so we may
assume that $K(h,k)(u(a',i))$ has already been defined. Note that we
have a 1-open $A, a$-box given by $K(h,k) \circ u$. Furthermore, note
that by applying the lower square of \eqref{eq:khknat} to each
$u(a',i)$ we have that $K(h,k) \circ u$ is an open box over
$k(y)$. Hence $(K(h,k) \circ u, k(y))$ is an element of $K g$, and so
we can take $K(h,k) (x)$ to be $(K(h,k) \circ u, k(y))$. Note that we can now
easily see that the lower square of \eqref{eq:khknat} holds
``locally'' at $x$.

Given $\langle a \rangle (u, y) \in K^+ f$, we define $K(h,k)$ to be 
$\langle a \rangle (K(h,k) \circ u, k(y))$, noting that we can use
lemma \ref{lem:nameabsext} to ensure we get a well defined function.

We can similarly define $K(h,k)$ on elements of
$K^\downarrow f$ and $K^- f$.

Finally note that one can easily show by induction that $K$ preserves
composition and identities.

\section{Monad and Comonad Structure}

\subsection{$\monadf{L}$ is a Comonad}

We define the comonad $\monadf{L}$ as follows. We need to define a
comultiplication morphism $\Sigma : L \rightarrow L^2$. Given a
morphism $f : X \rightarrow Y$, recall that $L f = \lambda_f : X
\rightarrow K f$. This means that $L^2 f$ is the morphism
$\lambda_{\lambda_f} : X \rightarrow K \lambda_f$. We will define
$\sigma_f$ such that $\Sigma_f$ will be the commutative square
\begin{equation}
  \label{eq:sigmadef}
  \begin{gathered}
    \xymatrix{ X \ar@{=}[r] \ar[d]_{\lambda_f} & X \ar[d]^{\lambda_{\lambda_f}} \\
      K f \ar[r]_{\sigma_f} & K \lambda_f
      }
  \end{gathered}
\end{equation}
We inductively define $\sigma_f : K f \rightarrow K \lambda_f$, by
taking it to be the identity on $X$ (note that this implies that
\eqref{eq:sigmadef} is commutative). Then given an element, $(u, y)$
of $K_n^\uparrow f$, send it to $(\sigma_f \circ u, (u,y))$ in $K
\lambda$. Given an element $\langle a \rangle (u, y)$ of $K^+ f$, send
it to $\langle a \rangle (\sigma_f \circ u, (u, y))$ (which is well
defined by lemma \ref{lem:nameabsext}), and similarly
for the remaining components of $K_n f$). Checking that $\monadf{L}$
is indeed a comonad amounts to checking the following three diagrams
commute:
\begin{gather}
  \begin{gathered}
    \xymatrix{ K f \ar@{=}[dr] & K \lambda_f \ar[l]_{\rho_{\lambda_f}}
     \\
    & K f \ar[u]_{\sigma_f}
  }
\end{gathered}  \label{eq:comonad1} \\
\begin{gathered}
  \xymatrix{ K f \ar@{=}[dr] & K \lambda_f \ar[l]_{K \Phi_f}
     \\
    & K f \ar[u]_{\sigma_f}
  }
\end{gathered} \label{eq:comonad2} \\
\begin{gathered}
  \xymatrix{ K \lambda_{\lambda_f} 
     & K \lambda_f  \ar[l]_{K \Sigma_f} \\
    K \lambda_f \ar[u]^{\sigma_{\lambda_f}} & K f \ar[u]_{\sigma_f} \ar[l]^{\sigma_f}
    }
\end{gathered} \label{eq:comonad3}
\end{gather}
All three can be verified by induction on rank. To illustrate, we
check the comultiplication law, \eqref{eq:comonad3} below. The case
when we are given an element of $X$ is clear, so it remains to check
the cases where we are given an element of the form $(u, y)$ or
$\langle a \rangle (u, y)$. We will just verify the case $(u, y)$; the
other case is similar.
\begin{align}
  K \Sigma_f \circ \sigma_f (u, y) &= K \Sigma_f (\sigma_f \circ u, (u,
  y)) \\
  &= (K \Sigma_f \circ \sigma_f \circ u, \sigma_f(u, y)) \\
  &= (K \Sigma_f \circ \sigma_f \circ u, (\sigma_f \circ u, (u, y))) \\
  \sigma_{\lambda_f} \circ \sigma_f (u, y) &=
  \sigma_{\lambda_f} (\sigma_f \circ u, (u, y)) \\
  &= (\sigma_{\lambda_f} \circ \sigma_f \circ u, (\sigma_f \circ u,
  (u, y))) \\
  &= (K \Sigma_f \circ \sigma_f \circ u, (\sigma_f \circ u,
  (u, y))) &\text{by induction on rank} \\
  &= K \Sigma_f \circ \sigma_f (u, y)
\end{align}

\subsection{$\monadf{R}$ is a Monad}

To define the monad $\monadf{R}$, we need to define a multiplication
map $\Pi : R^2 \rightarrow R$. Given $f : X \rightarrow Y$, $R f$ is
of the form $\rho_f : K f \rightarrow Y$, and so $R^2 f$ is of the
form $\rho_{\rho_f} : K \rho_f \rightarrow Y$. We will define $\pi_f$
so that $\Pi_f$ is the commutative diagram
\begin{equation}
  \label{eq:pidef}
  \begin{gathered}
    \xymatrix{ K \rho_f \ar[r]^{\pi_f} \ar[d]_{\rho_{\rho_f}}  & K f
       \ar[d]^{\rho_f} \\
      Y \ar@{=}[r] & Y }
  \end{gathered}
\end{equation}

We define $\pi_f$ inductively. If $x \in K \rho$ is already an element
of $K f$, then we can take $\pi_f x$ to be $x$. If $x \in K \rho_f$ is
of the form $(u, y)$, then note that $\pi_f \circ u$ is an open box
over $y$, and so we can take $\pi_f x$ to be $(\pi_f \circ u, y)$. To
check that this is a monad, we need to verify the commutativity of the
following diagrams.

\begin{gather}
  \begin{gathered}
    \xymatrix{ Kf \ar[r]^{\lambda_{\rho_f}} \ar@{=}[dr] & K \rho_f
      \ar[d]^{\pi_f} \\
      & K f }
  \end{gathered} \\
  \begin{gathered}
    \xymatrix{ K f \ar[r]^{K \Lambda_f} \ar@{=}[dr] & K \rho_f
      \ar[d]^{\pi_f} \\
      & K f }
  \end{gathered} \\
  \begin{gathered}
    \xymatrix{ K \rho_{\rho_f} \ar[r]^{K \Pi_f} \ar[d]_{\pi_{\rho_f}}
      & K \rho_f \ar[d]^{\pi_f} \\
      K \rho_f \ar[r]_{\pi_f} & K f }      
  \end{gathered}
\end{gather}
As before, each diagram can be checked by induction on rank.

\section{Kan Fibrations}

Every algebraic weak factorisation system gives rise to a weak
factorisation system $(\bar{\mathcal{L}}, \bar{\mathcal{R}})$ where
$\bar{\mathcal{L}}$ and $\bar{\mathcal{R}}$ are the retract closures
of $\monadf{L}$-maps and $\monadf{R}$-maps respectively (the fact that
this is indeed a weak factorisation system follows from standard
results in homotopical algebra). In fact $\bar{\mathcal{L}}$ and
$\bar{\mathcal{R}}$ can also be characterised as the class of maps
that admit $\monadf{L}$-coalgbra and $\monadf{R}$-algebra structures
when $\monadf{L}$ and $\monadf{R}$ are viewed as a copointed and
pointed endofunctor respectively, that is, when only the counit and
unit law are required and not the comultiplication and multiplication
law (see \cite[Lemma 2.8]{riehlams}). We will check that
$\bar{\mathcal{R}} = \mathcal{R}$ and show that this class is
precisely the class of Kan fibrations in the sense of \cite{pittssub},
and in fact $\monadf{R}$ algebra structures correspond exactly to Kan
filling operations.

\begin{lemma}
  \label{lem:conskanfill}
  Suppose that $f : X \rightarrow Y$ is a map in $\sub$ and that we
  are given an algebra structure for $f$ over $\monadf{R}$ regarded as
  a pointed endofunctor (that is, we only require the unit law and not
  the multiplication law for algebras). Then we can produce Kan
  filling operators $f \uparrow$ and $f \downarrow$ for $f$ in a
  canonical way.
\end{lemma}

\begin{proof}
  A pointed endofunctor algebra structure for $f$ is precisely a map
  $g$ satisfying the following commutative diagrams.
  \begin{gather}
    \begin{gathered}
      \xymatrix{ K f \ar[d]_{\rho_f}
        \ar[r]^g & X \ar[d]^f \\
         Y \ar@{=}[r] & Y
      }
    \end{gathered} \\
    \begin{gathered}
      \xymatrix{ X \ar@{=}[dr] \ar[r]^{\lambda_f} & Kf \ar[d]^g \\
        & X
        }
    \end{gathered}
  \end{gather}
  Let $(u, y)$ be an open box in $X$ over $y$ in $Y$. Note that we can
  regard $(u, y)$ as an element of $K f$. Define $f \uparrow (u, y)$
  to be $g((u,y))$. One can easily check that this is a filler for
  $(u,y)$ and that $f \uparrow$ respects the uniformity conditions by
  applying the above two diagrams and the fact that $g$ preserves
  permutations and substitutions (since it is a morphism in
  $\sub$). Similarly for $f \downarrow$.
  % TODO: FILL IN THE REMAINING DETAILS
\end{proof}

\begin{lemma}
  \label{lem:consalg}
  Suppose that we are given $f : X \rightarrow Y$ and uniform Kan
  filling operations $f \uparrow$ and $f \downarrow$ for $f$. We can
  produce an $\monadf{R}$-algebra structure on $f$ in a canonical way.
\end{lemma}

\begin{proof}
  Define $g : K f \rightarrow X$ by induction. For $k$ a rank 0
  element of $K f$, take $g(x)$ to be $x$.  For $(u, y)$ in
  $K^\uparrow f$, note that we may assume by induction that $g \circ
  u$ is defined and that it is an open box over $y$ in $X$. Define
  $g((u,y))$ to be $f \uparrow (g \circ u, y)$. Note that this is a
  morphism in $\sub$ (ie preserves permutations and substitutions) by
  applying the uniformity conditions. We need to check that the
  following diagrams commute.
  \begin{gather}
    \begin{gathered}
      \xymatrix{ X \ar@{=}[dr] \ar[r]^{\lambda_f} & Kf \ar[d]^g \\
        & X
        }
    \end{gathered} \\
    \begin{gathered}
      \xymatrix{ K f \ar[d]_{\rho_f}
        \ar[r]^g & X \ar[d]^f \\
         Y \ar@{=}[r] & Y
      }
    \end{gathered} \\
    \begin{gathered}
        \xymatrix{ K \rho_f \ar[r]^{\pi_f} \ar[d]_{K (g, 1)} & K f
          \ar[d]^g \\
          K f \ar[r]_g & X
        }
      \end{gathered}
  \end{gather}
  The first follows easily from the definition of $g$, and the
  remaining two can be shown by induction on rank.
\end{proof}

\begin{theorem}
  For $f : X \rightarrow Y$, $\monadf{R}$-algebra structures on $f$
  are in precise correspondence to pairs $f \uparrow$ and $f
  \downarrow$ of Kan filling operators.
\end{theorem}

\begin{proof}
  Observe that the construction in the proof of lemma \ref{lem:consalg} is
  a left inverse to the construction in the proof of lemma
  \ref{lem:conskanfill}, and is also a right inverse when restricted to
  (monad) $\monadf{R}$-algebra structures.
\end{proof}

\begin{theorem}
  \label{thm:fibdefs}
  Let $f : X \rightarrow Y$ be a morphism in $\sub$. Then the
  following are equivalent.
  \begin{enumerate}
  \item $f$ is an element of $\bar{\mathcal{R}}$
  \item $f$ admits a pointed endofunctor $\monadf{R}$-algebra
    structure
  \item $f$ is a Kan fibration
  \item $f$ admits a (monad) $\monadf{R}$-algebra structure
  \end{enumerate}
\end{theorem}

\begin{proof}
  For $(1)\Rightarrow (2)$, this follows easily from the fact that
  $(\bar{\mathcal{L}}, \bar{\mathcal{R}})$ is a wfs and in particular
  that $\mathcal{L} \pitchfork \bar{\mathcal{R}}$. (See \cite[Lemma
  2.8]{riehlams})

  For $(2) \Rightarrow (3)$, apply lemma \ref{lem:conskanfill}.

  For $(3) \Rightarrow (4)$, apply lemma \ref{lem:consalg}.

  Finally note that $(4) \Rightarrow (1)$ is trivial.
\end{proof}

\begin{remark}
  The implication $(2) \Rightarrow (4)$ in theorem \ref{thm:fibdefs}
  is to be expected, since it holds in general for awfs's generated with
  Garner's version of the small object argument (see 
  \cite[Lemma 2.30]{riehlams}).
\end{remark}

\section{The Generating Left Maps}

In \cite{garnersmallobject}, Garner shows how to define an awfs from a
generating diagram of $\monadf{L}$-maps. The resulting awfs is
\emph{algebraically free} over the diagram.  Since the construction
here is constructive variant of Garner's argument the same should be
true here. In this section we define a diagram based on box inclusions
and check that indeed the awfs is algebraically free over the diagram.

\subsection{Definition of the Diagram}

We now define the diagram of generating $\monadf{L}$-maps. This is
based on open box inclusions for cubical sets (as described by Huber
in \cite[Remark 3.9]{huberthesis}) translated into 01-substitution
sets.

Let $\mathcal{J}$ be the category defined as follows. The objects of
$\mathcal{J}$ are 4-tuples $(i, A, a, B)$, where $i \in 2$, $A$ and
$B$ are finite subsets of $\names$ and $a \in A$. A morphism $(i, A,
a, B) \rightarrow (i', A', a', B')$ is a pair $(f, g)$ such that $i =
i'$, $f : A' \rightarrow A$ is a bijection with $f(a') = a$ and $g :
B' \rightarrow B$ is a morphism in the category of names and
substitutions, ie a function $B' \rightarrow B \amalg 2$ which is
``injective where defined.'' (Note that $f$ and $g$ are in the
opposite direction to $(f, g)$).

\begin{definition}
  \label{def:2}
  Given a function $f : A \rightarrow \names \amalg 2$, say $f$ is
  \emph{injective where defined} if the function $f|_{f^{-1}(\names)}$
  is an injection.
\end{definition}

\begin{definition}
  \label{def:3}
  Given $A$ a finite subset of $\names$, define the \emph{standard
    $A$-cube}, $\square_A$, to be the set of functions $A \rightarrow
  \names \amalg 2$ that are injective where defined. We make
  $\square_A$ a 01-substitution set by defining
  \begin{equation}
    (\pi . f)(a) := 
    \begin{cases}
      \pi(f(a)) & f(a) \in \names \\
      f(a) & f(a) \in 2
    \end{cases}
  \end{equation}
  and
  \begin{equation}
    \label{eq:8}
    (f (a' := i))(a) :=
    \begin{cases}
      f(a) & f(a) \in 2 \\
      f(a) & f(a) \in \names \text{ and } f(a) \neq a' \\
      i & f(a) = a'
    \end{cases}
  \end{equation}
\end{definition}

Note that this gives a functor from the opposite of the category of
names and substitutions to $\sub$. Given $f : B \rightarrow A$, define
$\square_f : \square_A \rightarrow \square_B$ by sending $g \in
\square_A$ to $g \circ f$. (This corresponds to the functor to cubical
sets given by the Yoneda lemma)

\begin{definition}
  \label{def:1}
  Given $A$ a finite subset of $\names$ and $a \in A$, define the
  \emph{standard 1-open $A,a$-box}, $\sqcup_{A,a}$ in $\sub$ to be the
  subset of $\square_A$ consisting of $f$ such that for some $(a', i)
  \in A \times 2 \setminus \{(a, 1)\}$, $f(a') = i$. Note that this is
  a subobject of $\square_A$ by inclusion, in $\sub$.

  The \emph{standard 0-open $A,a$-box}, $\sqcap_{A,a}$, is defined
  analogously.
\end{definition}

Note that the morphism $\square_f : \square_A \rightarrow \square_B$
given above restricts to a morphism $\sqcup_{A, a} \rightarrow
\sqcup_{B, b}$ if $f(b) = a$.

We define a functor $J : \mathcal{J} \rightarrow \sub^\btwo$ as
follows. On an object $(1, A, a, B)$, define $J(1, A, a, B)$ to be the
map from $\sqcup_{A, a} \ast \square_B$ to $\square_A \ast \square_B$
defined by taking the product of the inclusion $\sqcup_{A,a}
\hookrightarrow \square_A$ with the identity on $\square_B$.

Given a morphism $(f, g) : (1, A, a, B) \rightarrow (1, A', a', B')$,
define $J(f,g)$ componententwise by composing $f$ and $g$ with
elements of $\square_A$ and $\square_B$ respectively.

\subsection{Algebraic Freeness}

We will show that the awfs is algebraically free on $J$. This means we
need to find $\eta : \mathcal{J} \rightarrow \lmap$ over $\sub^\btwo$
such
% NOTE: GARNER SAYS THIS IS A MORPHISM OVER C (NOT C^2)
that the following is an isomorphism of categories (see
\cite[Definition 3.9]{garnersmallobject}).
\begin{equation}
  \label{eq:shouldbeiso}
  \rmap \stackrel{\operatorname{lift}}{\longrightarrow} \lmap^\pitchfork
    \stackrel{\eta^\pitchfork}{\longrightarrow} \mathcal{J}^\pitchfork
\end{equation}

We require that $\eta(1,A, a, B)$ is an $\monadf{L}$-coalgebra
structure on $J(1,A, a, B)$. We now fix $(1,A, a, B)$ and refer to
$J(1,A,a,B)$ as $\iota$. To find $\eta(1,A,a,B)$ we need to find $h :
\square_A \ast \square_B \rightarrow K \iota$ satisfying the following
commutative diagrams.
\begin{gather}
  \begin{gathered}
    \xymatrix{ \sqcup_{A,a} \ast \square_B \ar@{=}[r] \ar[d]_\iota &
      \sqcup_{A,a} \ast \square_B \ar[d]^{\lambda_\iota} \\
      \square_A \ast \square_B \ar[r]^h & K \iota }
  \end{gathered} \label{eq:coalg2morph} \\
  \begin{gathered}
    \xymatrix{ \square_A \ast \square_B \ar@{=}[dr] \ar[r]^h
      & K \iota \ar[d]^{\rho_\iota} \\
      & \square_A \ast \square_B }
  \end{gathered} \label{eq:coalgcounit} \\
  \begin{gathered}
    \xymatrix{ \square_A \ast \square_B \ar[r]^h \ar[d]_h & K \iota
      \ar[d]^{K(1, h)} \\
      K \iota \ar[r]^{\sigma_\iota} & K \lambda_\iota }
  \end{gathered} \label{eq:coalgcomult}
\end{gather}

If $(f,g) \in \square_A \ast \square_B$ is already an element of
$\sqcup_{A,a} \ast \square_B$, then define $h(f,g)$ to be $(f,
g)$. Note that this precisely ensures that \eqref{eq:coalg2morph}
commutes. 

If $f$ is defined everywhere on $A$, then let $u$ be the 1-open $f(A),
f(a)$-box defined for $a'' \in f(A)$ by
\begin{equation}
  \label{eq:hdefkanfill}
  u(a'', i) := f \, (a'' := i)
\end{equation}
We then define
\begin{equation}
  \label{eq:13}
  h(f, g) := (u, (f, g))
\end{equation}

If $f$ is defined everywhere on $A \setminus a$ and $f(a) = 1$, then
let $b$ be a fresh variable, define $f'$ as follows
\begin{equation}
  \label{eq:11}
  f'(a') :=
  \begin{cases}
    f(a') & a' \in A \setminus a \\
    b & a' = a
  \end{cases}
\end{equation}
Then define $u'$ as for $u$ in \eqref{eq:hdefkanfill}, but with $f'$
in place of $f$, and we can now define $h(f, g) \in K^+ \iota$ to be
\begin{equation}
  \label{eq:12}
  h(f, g) := \langle b \rangle (u', (f', g))
\end{equation}

This completes the definition of $\eta(1,A,a,B) = h$. The
commutativity of \eqref{eq:coalgcounit} is clear by definition and
\eqref{eq:coalgcomult} can easily be checked. The case for $0$-open
boxes is similar. We now show that \eqref{eq:shouldbeiso} is an
isomorphism by exhibiting an inverse.

The elements of $\mathcal{J}^\pitchfork$ are of the form $(g, \phi)$
where $g : X \rightarrow Y$ is a morphism in $\sub$ and $\phi$ is
lifting data for $g$ against $J$. We will use $\phi$ to define Kan
filling operators on $g$ and apply lemma \ref{lem:conskanfill} to get
an $\monadf{R}$-algebra structure on $g$.

Let $(u, y)$ be an open box in $X$ over $y$ in $Y$. Let $u$ be a
1-open $A, a$-box, and let $C$ be a finite support for $u$. Let $B :=
C \setminus A$, and note that $B$ is a finite set. For any $(a', i)
\in A \times 2 \setminus (1,a)$, we can define $\overline{(a',i)} \in
\sqcup_{A,a}$ by 
\begin{equation}
  \overline{(a',i)}(a'') := 
  \begin{cases}
    i & a'' = a' \\
    a'' & a'' \neq a' 
  \end{cases}
\end{equation}
Now note that there is a unique morphism $\tilde{u} : \sqcup_{A,a}
\rightarrow X$ such that for each $(a',i) \in A \times 2 \setminus (a,
1)$, $\tilde{u}(\overline{(a',i)}) = u(a',i)$. Also note that there is
a unique morphism $\tilde{y} : \square_A \ast \square_B \rightarrow Y$
such that $\tilde{y}(1_A, 1_B) = y$. (Both of these results can be
viewed as translating the Yoneda lemma from cubical sets to
01-substitution sets). These maps together make a commutative square:
\begin{gather}
  \begin{gathered}
    \xymatrix{ \sqcup_{A,a} \ast \square_B \ar[r]^{\tilde{u}}
      \ar@{^(->}[d] & X \ar[d]^g \\
      \square_A \ast \square_B \ar[r]^{\tilde{y}} & Y }
  \end{gathered}
\end{gather}
Then, applying the lifting data $\phi$ yields a diagonal filler $j :
\square_A \ast \square_B \rightarrow X$. One can check that $j(1_A,
1_B)$ is a filler for $(u,y)$.

To show that this gives a Kan filling operation, we first need to
check that it is well defined. So let $C'$ be another finite support
for $u$. Since the intersection of two finite supports is also a
finite support, we may assume without loss of generality that $C'
\subseteq C$. Hence also $B' \subseteq B$ and the inclusion $i : B'
\hookrightarrow B$ induces a morphism $(1, A, a, B) \rightarrow (1, A,
a, B')$ in $\mathcal{J}$. Applying the coherence condition for $\phi$
to $i$, then shows that the fillers we get using $B$ and using $B'$
are equal.

Similar arguments show that the filling operation we get satisfies the
uniformity conditions.

Exactly the same argument allows us to construct a filling operator
for 0-open boxes. Applying lemma \ref{lem:conskanfill} then gives us
an $\monadf{R}$-algebra structure on $g$.

Finally, one can check that this is in fact an inverse to the map
\eqref{eq:shouldbeiso}.

\begin{remark}
  We can use this characterisation to verify that for any map $f$ in
  the category of cubical sets, the fibration structures on $f$
  correspond precisely to the fibration structures on the image of $f$
  in $\sub$ under Pitts' equivalence in \cite{pittsnompcs}. This is
  because Huber gave in \cite[Remark 3.9]{huberthesis} a
  characterisation of fibrations that can be easily seen as showing
  the fibrations are cofibrantly generated. Since it can be easily
  checked that the generating diagram in this section is (naturally
  isomorphic to) the image of the generating diagram of left maps in
  \cite[Remark 3.9]{huberthesis}, we deduce that the resulting
  fibration structures are also the same (up to isomorphism).
\end{remark}

\section{Path Objects}

As shown by Awodey and Warren in \cite{awodeywarren}, identity
types in type theory can be implemented using path objects, which are
defined as follows.

\begin{definition}
  Let $(\mathcal{L}, \mathcal{R})$ be a wfs and let $f : X
  \rightarrow Y$ be a map in $\mathcal{R}$. A \emph{path object} on
  $f$ is a factorisation of the the diagonal map
  $\Delta : X \rightarrow X \times_Y X$ as a map in $\mathcal{L}$
  followed by a map in $\mathcal{R}$.
\end{definition}

Note that we can trivially generate path objects using the awfs
structure itself. For any fibration $f : X \rightarrow Y$, if $\Delta
: X \rightarrow X \times_Y X$ is the diagonal map, then $K \Delta$ is
a path object. However, in order to implement identity types it is
necessary for the path objects to be stable under pullback (see the
statement of \cite[Theorem 3.1]{awodeywarren}). 

This is not the case
for path objects generated using the awfs, as we prove below. The
basic idea here is that given a continuous function $f : X \rightarrow
Y$, for each point $x \in X$ and each path $p$ in $Y$ with $p(0) =
f(x)$, we freely added a path $\tilde{p}$ in $X$ over $p$ with
$\tilde{p}(0) = x$. However, our construction also adds the other
endpoint of $\tilde{p}$, $\tilde{p}(1)$. In our definition of the
awfs, this can be seen explicitly as elements of the $K^+ f$ component
(and similarly the $K^-$ component). The added endpoint $\tilde{p}(1)$
contains ``data'' about the whole path $p$, but lies entirely in the
fibre of $p(1)$. Hence the pullback along the map $1 \rightarrow Y$
given by $p(1)$ does not preserve the awfs.

\begin{theorem}
  \label{thm:1}
  There is a fibration $f : X \rightarrow Y$ and a map $g : Z
  \rightarrow Y$ such that, writing $g^\ast(f)$ for the pullback of
  $f$ along $g$, there is no map $K \Delta_{g^\ast(f)} \rightarrow K
  \Delta_f$ making the following diagram a pullback.
  \begin{equation}
    \label{eq:18}
    \begin{gathered}
      \xymatrix{
        K \Delta_{g^\ast(f)} \ar[r] \ar[d]_{\rho_{\Delta_{g^\ast(f)}}} &
        K \Delta_f \ar[d]_{\rho_{\Delta_f}} \\
        g^\ast(X) \times_{Z} g^\ast(X) \ar[r] & X \times_Y X }
    \end{gathered}
  \end{equation}
\end{theorem}

\begin{proof}
  Let $X = Y = N(\rint)$ (in fact any non trivial complete metric
  space will do), let $f$ be the identity on $N(\rint)$, and let $Z :=
  1$ and $g := \lambda x.0$.

  In particular, we have $X \times_Y X = X \times_X X \cong X$ and
  $g^\ast(X) \times_Z g^\ast(X) \cong 1$, so by functoriality of $K$,
  it suffices to show that $K {1_1}$ (by $1_1$ we mean the identity on
  the terminal object) is not isomorphic to $1 \times_X K 1_X$.

  $1 \times_X K 1_X$ contains an uncountable subset $U$, consisting of
  elements of the form $(\ast, \langle a \rangle (u, h))$ where $\ast$
  is the only element of $1$, $a \in \names$, $u$ is a $1$-open
  $(\{a\}, a)$-box and $h$ is any function $\rint \rightarrow \rint$
  such that $h(0) = u(a, 0)$ and $h(1) = 0$, which we can view as an
  element of $N(\rint)$ via lemma \ref{lem:ctsfunctionsnerve}. To show
  $(\ast, \langle a \rangle(u, h))$ belongs to $1 \times_X K 1_X$, we
  need that $\rho_f(\langle a \rangle (u, h)) = 0$, but this is the
  case because $\rho_f(\langle a \rangle (u, h)) = h(a := 1) = h(1) =
  0$. Also note that in general, for any nominal sets, if $\langle a
  \rangle x = \langle a \rangle x'$ then $x = x'$.

  However $K 1_1$ is countable, since it is a countable union of
  countable sets.

  (For this argument to work constructively, note that $U$ is an
  inhabited decidable subset of $1 \times_X K 1_X$ and such subsets of
  countable sets are countable. Furthermore, assuming countable
  choice, we can show a countable union of countable sets is countable
  and use Cantor's diagonal argument to show $U$ is uncountable.)
\end{proof}

\subsection{Name Abstraction}
The cubical sets used in \cite{bchcubicalsets} to implement identity
types correspond in 01-substitution sets to name abstraction.

\begin{definition}
  Given $f : X \rightarrow Y$, define the \emph{name abstraction over
    $f$}, $[\names]_f X$ as follows:
\begin{equation}
  \label{eq:6}
  [\names]_f X := \{ \langle a \rangle x \in [\names] X \;|\; a \#
  f(x) \}
\end{equation}
\end{definition}

One can adapt the proof in \cite{bchcubicalsets} that identity types
are Kan fibrations to show that the projection map $[\names]_fX
\rightarrow X \times_Y X$ is an $\monadf{R}$-map. 
%[TODO: CHECK THIS]

In order to show that $[\names]_fX$ is a path object on $X$ over $f$,
it remains only to show that the reflexivity map $r : X \rightarrow
[\names]_fX$ has the left lifting property with respect to every
$\monadf{R}$-map. ($r(x)$ is defined to be Fresh $a$ in $\langle a
\rangle x$.)

We will see later (corollary \ref{cor:homotopyispathobj}) that
assuming classical logic this is a path object. However, we show now
that this cannot be done constructively. In particular, we will show
that one cannot show constructively that $r : X \rightarrow [\names]_f
X$ is always a left map.

\begin{lemma}
  \label{lem:cofibimdec}
  Suppose that $i : U \rightarrow V$ has the left lifting property
  with respect to every $\monadf{R}$-map. Then $i$ has decidable
  image. (That is, every element of $V$ either lies in the image of
  $i$ or does not lie in the image of $i$).
\end{lemma}

\begin{proof}
  First, note that we can show that the map
  $\lambda_i : U \rightarrow K i$ has decidable image. This is because
  rank is well defined by lemma \ref{lem:rank}, but the image of
  $\lambda_i$ is precisely the subset of $K i$ of rank $0$, and every
  natural number is either equal to $0$ or greater than $0$.

  Now by assumption, $i : U \rightarrow V$ has the left
  lifting property with respect to the $\monadf{R}$-map $\rho_r$,
  giving us a diagonal filler, $j$, in the following diagram:
  \begin{equation}
    \label{eq:4}
    \begin{gathered}
      \xymatrix{ U \ar[r]^{\lambda_i} \ar[d]_i & K i \ar[d]^{\rho_i} \\
        V \ar@{=}[r] \ar@{.>}[ur]|j & V
      }
    \end{gathered}
  \end{equation}

  However, note that it easily follows from this diagram that any $v
  \in V$ lies in the image of $i$ if and only if $j(v)$ lies in the
  image of $\lambda_i$. Since we have checked that $\lambda_i$ has
  decidable image, it follows that $i$ must also have decidable
  image.
\end{proof}

\begin{theorem}
  \label{thm:nonarlm}
  It cannot be proved constructively that for every fibrant
  01-substitution set, $X$, the reflexivity morphism $r : X
  \rightarrow [\names]X$ has the left lifting property with respect to
  every $\monadf{R}$-map.
\end{theorem}

\begin{proof}
  We will show this by assuming $r$ does have this property and
  deriving the weak limited principle of omniscience (WLPO), that is,
  for any $f : \nat \rightarrow 2$ the statement ``$f(n) = 0$ for all
  $n$'' is either true or false. (In particular this implies the
  existence of noncomputable functions, and so is not provable
  constructively, even if we assume dependent choice, eg by
  considering the realizability model in \cite{rathjen06})

  Let $f : \mathbb{N} \rightarrow 2$. Define $X := N(\rint)$, the
  nerve of $\rint$. Since $\rint$ is complete, we have that $N(\rint)$
  is fibrant by proposition \ref{prop:nervefibrant}. 

  We construct a continuous function $\bar{f} : \rint \rightarrow
  \rint$. Given $x \in \rint$ we will construct a Cauchy sequence
  $(y_n)_{n < \omega}$. Let $n < \omega$. If $f(m) = 0$ for all $m <
  n$, then let $y_n := 0$. If there is $m < n$ such that $f(m) = 1$,
  then there is a least $m_0$ such that $f(m_0) = 1$. Let $y_n :=
  2^{-m_0} q$, where $q$ is chosen such
  that $|q - x| < 2^{-n}$. Note that in classical logic $\bar{f}$
  could have been defined
  \begin{equation}
    \label{eq:3}
    \bar{f}(x) =
    \begin{cases}
      0 & f(n) = 0 \text{ for all } n \\
      2^{-n} x & \text{there is } n
      \text{ least such that } f(n) = 1
    \end{cases}
  \end{equation}

  Now note that we can view $\bar{f}$ as an element of $X$ dependent
  on (at most) a single name, say, $a$ by lemma
  \ref{lem:ctsfunctionsnerve}. So we have that $\langle a \rangle
  \bar{f} \in [\names]X$.

  We then have that $\langle a \rangle \bar{f}$ lies in the image of
  $r$ if and only if $a$ is fresh for $\bar{f}$, which is the case
  precisely when $\bar{f}$ is constant with respect to $a$, which in
  turn happens precisely when $f$ is constantly $0$. But we showed in
  lemma \ref{lem:cofibimdec} that the image of $r$ must be
  decidable. Therefore, WLPO follows as required.
\end{proof}

\subsection{Labelled Name Abstraction}

In this section we give a new construction of path object that is
valid constructively and preserved by pullback. The basic idea is to
add ``labels'' to the ``name abstraction'' identity types. Any side of
a cube that has been ``labelled'' in guaranteed to be degenerate
(although there may be additional degenerate sides that are not
labelled). In order to make the notation and proofs slightly easier we
use some notions from the construction of the awfs. In particular, we
have already checked that $K \Delta$ is a 01-substitution set, so we
know, for instance that ``commutativity of substitution'' holds, ie $z
(a := i) (a' := i') = z (a' := i') (a := i)$. A full proof that these
objects can be used to model identity types will be left for another
paper, using ideas from algebraic model structures. For now we simply
give a definition and verify that they are path objects and stable
under pullback.

\begin{definition}
  \label{def:4}
  Let $f : X \rightarrow Y$ be a fibration and let $a \in
  \names$. Define the subset $P^0_Y X$ of $K \Delta$ of
  \emph{pre-normal forms in direction $a$} inductively as follows.
  \begin{enumerate}
  \item If $x \in X$ and $a \# x$ then $a$ is a pre-normal form in
    direction $a$.
  \item If $u$ is a $1$-open $A,a$-box over $(x_1, x_2)$, $a \# x_1$,
    $u(a,0) = x_1$ and for every $a', i \in (A \setminus a) \times 2$,
    $u(a',i)$ is a pre-normal form in direction $a$, then $(u, (x_1,
    x_2))$ is a pre-normal form in direction $a$.
  \end{enumerate}
  Note that $P^0_Y X$ is closed under permutations, and every
  substitution of the form $(a' := i)$ where $a' \neq a$.

  An element of $K \Delta$ is a \emph{normal form} if it is equal to
  $z (a := 1)$ where $z$ is some pre-normal form in direction $a$.

  We refer to the set of normal forms as $P_Y X$. Note that it is
  closed under permutations and substitutions, and so we can view it
  as a 01-substitution set in the natural way.
\end{definition}

\begin{lemma}
  \label{lem:bindinj}
  Let $z, z'$ be pre-normal forms in direction $a$. If $z (a := 1) =
  z' (a := 1)$ then $z = z'$.
\end{lemma}

\begin{proof}
  Note that if $z$ is an element of $X$ with $a \# z$, then $z (a :=
  1) = z$. If $z$ is an open box in direction $a$, then $z (a := 1)$
  is an element of $K^+ \Delta$.

  Therefore, if $z (a := 1) = z' (a := 1)$, then either $z$ and $z'$
  are both elements of $X$, or they are both elements of $K^\uparrow
  \Delta$.

  In the former case, we have $z (a := 1) = z$ and $z' (a := 1) = z'$
  (since $a \# z$ and $a \# z'$) and so $z = z'$ as required.

  In the latter case, we have that $z$ and $z'$ are of the form $(u,
  (x_1, x_2))$ and $(u', (x_1', x_2'))$ respectively. Then $z (a :=
  1)$ and $z' (a := 1)$ are the elements of $K^+ \Delta$, $\langle a
  \rangle (u, (x_1, x_2))$ and $\langle a \rangle (u', (x_1',
  x_2'))$. Since $\langle a \rangle (u, (x_1, x_2)) = \langle a
  \rangle (u', (x_1', x_2'))$, we have $(u, (x_1, x_2)) = (u', (x_1',
  x_2'))$ as required.
\end{proof}

\begin{corollary}
  \label{cor:pnfunique}
  For every normal form $w$ and every $a \in \names$ fresh, there is a
  unique pre-normal form $z$ in direction $a$ such that $z (a := 1) =
  w$.
\end{corollary}

Unlike the path objects from the awfs, this construction is stable
under pullback, as we show below.
\begin{theorem}
  \label{thm:2}
  Given $f : X \rightarrow Y$, write $p_f$ for the map $P_Y X
  \rightarrow Y$ given by composition of $\rho_\Delta$ with the
  canonical map $X \times_Y X \rightarrow Y$. Suppose we are given a
  commutative square
  \begin{equation}
    \label{eq:20}
    \begin{gathered}
      \xymatrix{ U \ar[r]^h \ar[d]_g & X \ar[d]_f \\
        V \ar[r]^k & Y}
    \end{gathered}
  \end{equation}

  Then we can define a map $P_V U \rightarrow P_Y X$ making the
  following square commute.
  \begin{equation}
    \label{eq:21}
    \begin{gathered}
      \xymatrix{ P_V U \ar[r] \ar[d]_{p_g} & P_Y X \ar[d]_{p_f} \\
        V \ar[r]^k & Y}
    \end{gathered}
  \end{equation}

  Furthermore, if \eqref{eq:20} is a pullback square, then so is
  \eqref{eq:21}. 
\end{theorem}

\begin{proof}
  The morphisms $P_V U \rightarrow P_YX$ can easily be defined by
  using corollary \ref{cor:pnfunique} to lift to $P^0$, and then
  working by induction on the construction of $P^0$.

  Now assume that \eqref{eq:20} is a pullback. To check that
  \eqref{eq:21} is a pullback, it suffices to show that the canonical
  map $P_V U \rightarrow V \times_Y P_Y X$ is an isomorphism. So we
  need to show that given $(v, x) \in V \times_Y P_Y X$, there is a
  unique $u \in P_V U$ mapped to $(v, x)$. This is done by induction on the
  construction of $P^0$. If $x \in X$, then we can use the fact that
  \eqref{eq:20} is a pullback. We now deal with the case where $x$ is
  of the form $\langle a \rangle(w, (x_1, x_2))$. The key point here
  is that we required in definition \ref{def:4} that $a \#
  x_1$. Hence, writing $\pi$ for the canonical map $X \times_Y X
  \rightarrow Y$, we can deduce
  \begin{align}
    p_f (w, (x_1, x_2)) &= \pi(\rho_\Delta (w, (x_1, x_2))) \\
    &= \pi(x_1 (a := 1), x_2 (a := 1)) \\
    &= f(x_1 (a := 1)) \\
    &= f(x_1) & \text{(since } a \# x_1 \text{)}
  \end{align}
  Since $(v, x) \in V \times_Y P_Y X$ we must also have $p_f((w, (x_1,
  x_2))) = k(v)$, and so $k(v) = f(x_1) = f(x_2)$. But now using that
  \eqref{eq:20} is a pullback, we have uniquely specified $u_1, u_2$
  in $U$ such that $g(u_1) = g(u_2) = v$, $h(u_1) = x_1$ and $h(u_2) =
  x_2$. Applying the inductive hypothesis, we also have a uniquely
  determined open box over $w$, so there is a unique element of $P_V
  U$ mapped to $\langle a \rangle(w, (x_1, x_2))$, as required.
\end{proof}

We now show that this construction does give us path
objects.

\begin{theorem}
  The restriction of $\rho_\Delta : K \Delta \rightarrow X \times_Y X$
  to $P_Y X$ is a fibration.
\end{theorem}

\begin{proof}
  Let $v$ be a $1$-open $A, a$-box in $P_Y X$ over $(x_1, x_2)$. Let
  $b$ be a fresh variable. By applying corollary \ref{cor:pnfunique}
  we get a 1-open $A,a$-box $v'$ over $P_Y^0 X$ satisfying the
  adjacency conditions by lemma \ref{lem:bindinj} and such that for
  each $(a',i) \in A \times 2 \setminus (a,1)$, $v'(a',i)$ is a
  pre-normal form in direction $b$. Note that $\pi_2 \circ \rho_\Delta
  \circ v'$ is an open box in $X$ (where $\pi_2$ is the second
  projection). Let $A'$ be $A \cup \{b\}$. We extend
  $\pi_2 \circ \rho_\Delta \circ v'$ to a $1$-open $A', a$-box $v''$
  by setting $v''(b,0) = x_1$ and $v''(b,1) = x_2$.

  Let $A''$ be $A' \setminus a$. We define a $1$-open $A'',b$-box, $w$
  as follows. Define $w(b,0)$ to be $x_1 (a := 1)$. For $(a', i) \in
  A' \times 2$, define $w(a',i)$ to be $v'(a',i) (a := 1)$. Note that
  $(w, (x_1 (a := 1), f^+ v''))$ is a pre-normal form.

  We now form a $1$-open $A', b$-box, $w'$ as follows. Set $w'(b, 0)$
  to be $x_1$. Set $w'(a, 1)$ to be $(w, (x_1 (a := 1), f^+ v''))$. For
  $(a',i) \in A \times 2 \setminus (a,1)$, define $w'(a',i)$ to be
  $v'(a',i)$.

  Finally this allows to define the Kan filler of $v$ to be
  \begin{equation}
    \label{eq:14}
    (w', (x_1, f \uparrow v'')) (b := 1)
  \end{equation}
\end{proof}

We now show that $r$ is a left map. The key point in the lemmas below
is to ensure that we always treat elements of the $X$ component in
$P^0_Y X$ as degenerate paths, while also ensuring substitutions are
preserved.

\begin{lemma}
  \label{lem:makehomotopy}
  There is a nominal set morphism $h : P^0_Y X \ast \names \rightarrow
  X$ such that given $z$ in direction $a$, with $b$ a fresh name, and
  $a'$ a name with $a' \neq a$ and $a' \neq b$, we have $h(z (a' :=
  i), b) = h(z, b) (a' := i)$, $(h(z, b) (b := 0),
  h(z, b) (b := 1)) = \rho_\Delta z$, and for $x \in X$, $h(x, b) = x$.
\end{lemma}

\begin{proof}
  We define $h$ by induction on the construction of $P^0_Y X$. We
  define $h(x, b)$ to be $x$. Suppose we are given a pre-normal form
  of the form $(u, (x_1, x_2))$, where $u$ is a $1$-open $A,a$-box and
  $b$ a fresh name. Let $A' := A \cup \{b\}$. We form a $1$-open $A',
  a$-box, $v$ in $X$ as follows. If $(a', i)$ is an element of $A
  \times 2 \setminus (a,1)$, then we may assume by induction that
  $h(u(a',i), b)$ has already been defined. Let $v(a',i)$ be
  $h(u(a',i), b)$. Let $v(b,0)$ be $x_1$ and let $v(b,1)$ be
  $x_2$. Now define $h((u, (x_1, x_2)), b)$ to be $f \uparrow v$.
\end{proof}

To illustrate $h$, consider when $x_2$ is a path with $x_1$ as one of
the endpoints. In this case, $h$ is a homotopy from the constant path
at $x_1$ to the path $x_2$. If we are given an element $x$ of the $X$
component of $P^0_Y X$, then we think of it as the degenerate path
from $x$ to itself, and then return the degenerate homotopy from an
endpoint of the degenerate path to the path (and in this case all of
the objects mentioned happen to be equal as elements of $X$).

Next, in the lemma below we produce a homotopy from a path to its
endpoint, this time viewing paths are elements of $P^0_Y X$ rather
than elements of $X$, so we have some extra structure to take care of.

\begin{lemma}
  \label{lem:makehomotopy2}
  There is a nominal set morphism $k : P^0_Y X \ast \names \rightarrow
  P^0_Y X$ such that if $a$ is the direction of $z \in P^0_Y X$, $b$
  is a fresh name and $a'$ is a name with $a' \neq a$ and $a' \neq b$,
  then $k(z (a' := i), b) = k(z, b) (a' := i)$, and for $x \in X$ we
  have $k(x, b) = x$ and such that if $z \in P^0_Y X$ and $\rho_\Delta
  (z) = (x_1, x_2)$, then $k(z, b) (b := 0) = x_1$ and $k(z, b) (b :=
  1) = z$. $k$ will also be ``direction preserving.''
\end{lemma}

\begin{proof}
  We define $k(x, b)$ to be $x$ for $x \in X$. Now suppose we
  are given a pre-normal form of the form $(u, (x_1, x_2))$, where $u$
  is a $1$-open $A,a$-box and $b$ a fresh name. Let $A' := A \cup
  \{b\}$ as before. We form a $1$-open $A', a$ box $v$. Given $(a',i)
  \in A \times 2 \setminus (a, 1)$, set $v(a',i)$ to be $k(u(a',i),
  b)$. Set $v(b,0)$ to be $x_1$, and set $v(b,1)$ to be $(u, (x_1,
  x_2))$. Finally, define $k((u, (x_1, x_2)), b)$ to be $(v, (x_1,
  h((u, (x_1, x_2)), b)))$, where $h$ is as in lemma
  \ref{lem:makehomotopy}.
\end{proof}

\begin{theorem}
  The inclusion $r : X \rightarrow P_Y X$ is an element of
  $\bar{\mathcal{L}}$.
\end{theorem}

\begin{proof}
  We need to define a copointed endofunctor coalgebra on $r$.

  We first define a morphism $l : P^0_Y X \ast \names \rightarrow K
  r$, ensuring that if $z \in P^0_Y X$ is in direction $a$, then $a$
  is fresh for $l(z, b)$. As before, we ensure that $l$ is a nominal
  set morphism that preserves substitutions $(a' := i)$ for $a' \neq
  a, b$. We will also ensure that $\rho_r(l(z, b)) = k(z, b) (a :=
  1)$.
  
  Define $l(x, b)$ to be $x$. Given $(u, (x_1, x_2))$ and $b$,
  where $u$ is an $A, a$-box, let $A' := (A \setminus a) \cup
  \{b\}$. We define a $1$-open $A', b$-box, $v$ as follows. Given
  $(a',i) \in (A \setminus a) \times 2$, define $v(a',i)$ to be
  $l(u(a',i), b)$. Define $v(b, 0)$ to be $x_1$. Then define $l((u,
  (x_1, x_2)), b)$ to be $(v, k((u, (x_1, x_2)),
  b) (a := 1))$, where $k$ is as in lemma \ref{lem:makehomotopy2}.

  This now allows us to define the coalgebra map $c: P_Y X \rightarrow
  K r$ by defining $c(z (a := 1))$ (for $z$ a prenormal form in
  direction $a$) to be $l(z, b) (b := 1)$, where $b$ is any fresh
  variable (so $c((u, (x_1, x_2)) (a := 1))$ will be an element of
  $K^+ r$). Note that we have ensured throughout that given an element
  $x$ of $X$, $c(x) = x$.  Finally to show that this is a coalgebra
  structure we verify the counit law below.
  \begin{align}
    \label{eq:1}
    \rho_r (c(z (a := 1))) &= \rho_r (l(z,b) (b := 1)) \nonumber \\
    &= \rho_r (l(z, b)) (b := 1) \nonumber \\
    &= k(z, b) (a := 1) (b := 1) \nonumber \\
    &= k(z, b) (b := 1) (a := 1) \nonumber \\
    &= z (a := 1)
  \end{align}
\end{proof}

\begin{corollary}
  \label{cor:homotopyispathobj}
  Assume that excluded middle holds. Then $[\names]_f X$ is a path
  object.
\end{corollary}

\begin{proof}
  We write $r'$ for the reflexivity map $X \rightarrow [\names]_f X$.
  
  We will first define a map $g: [\names]_f X \rightarrow P_Y X$ such
  that the following diagram commutes.
  \begin{equation}
    \begin{gathered}
      \xymatrix{ X \ar@{=}[r] \ar[d] & X \ar[d] \\
        [\names]_f X \ar[r]^g \ar[d] & P_Y X \ar[d] \\
        X \times_Y X \ar@{=}[r] & X \times_Y X
      }
    \end{gathered}
  \end{equation}

  Define
  \begin{equation}
    \label{eq:10}
    X_0 := \{ (x, a) \;|\; a \# f(x) \}
  \end{equation}
  We first define a morphism $g_0 : X_0 \rightarrow P^0_Y X$.
  Applying excluded middle, we have that every $x \in X$ has a least
  finite support, $\supp(x)$. We define $g_0(a, x)$ by induction on
  $|\supp(x)|$. If $a \# x$, then define $g_0(a, x) := x$. Otherwise,
  let $A := \supp(x)$, and note that $a \in A$. We define a $1$-open
  $A,a$-box $v$ as follows. Let $v(a, 0) := x (a := 0)$. For $(a',i)
  \in (A \setminus a) \times 2$, we may assume by induction that
  $g_0(\langle a \rangle x (a' := i))$ has already been defined. Let
  $v(a',i) := g(a, x (a' := i))$. We then define $g_0(a, x)$ to be $(v,
  (x (a := 0), x))$.

  We now define $g(\langle a \rangle x)$ to be $g_0(a, x) (a := 1)$.

  We also have a morphism $P_Y X \rightarrow [\names]_f X$ given by
  projection. This induces a morphism $K r \rightarrow K r'$, which
  combined with the result before and the coalgebra structure on $r$
  gives a commutative diagram.
  \begin{equation}
    \begin{gathered}
      \xymatrix{ X \ar@{=}[r] \ar[d]_{r'} & X \ar@{=}[r] \ar[d]_r & X
        \ar[d]_{\lambda_r}
        \ar@{=}[r] & X \ar[d]_{\lambda_{r'}}
        \\
        [\names]_f X \ar[r]^g \ar[d] & P_Y X \ar[d] \ar[r]^c & K r
        \ar[d] \ar[r] & K r' \ar[d] \\
        X \times_Y X \ar@{=}[r] & X \times_Y X \ar@{=}[r] & X \times_Y
        X \ar@{=}[r] & X \times_Y X }
    \end{gathered}
  \end{equation}
  However, this easily gives us a coalgebra structure on $r'$.
\end{proof}

\section{Acknowledgements}
I am grateful for discussions with Nicola Gambino, Simon Huber and
Christian Sattler, which were very helpful with this work.

\bibliographystyle{abbrv}
\bibliography{mybib}{}

\end{document}